\documentclass[11pt]{article}
\usepackage{amssymb}
\usepackage{amsthm}
\usepackage{amsmath}
\usepackage{fullpage,hyperref,algorithmic}
\usepackage{graphicx,url,amsmath,amsthm,amsfonts,amssymb,subfigure,bbm}
\usepackage{pdfsync}

\newtheorem{theorem}{Theorem}[section]
\newtheorem{lemma}[theorem]{Lemma}

\newcommand{\pr}{\mathbb{P}}

\newcommand{\jnote}[1]{}

\newcommand{\E}{{\mathbb E}}

\newcommand{\dist}{\mathsf{dist}}

\newcommand{\remove}[1]{}

\newcommand{\e}{\varepsilon}

\newcommand{\Aut}{\mathsf{Aut}}

\def\P{\mathbb{P}}
\begin{document}

\title{{\bf A Gaussian upper bound for \\ martingale small-ball probabilities}}

\author{James R. Lee\thanks{Computer Science, University of Washington.  Partially supported by NSF grant CCF-1217256.} \and Yuval Peres\thanks{Microsoft Research} \and Charles K. Smart\thanks{Mathematics, MIT}}

\date{}

\maketitle

\begin{abstract}
Consider a discrete-time martingale $\{X_t\}$ taking values in a Hilbert space $\mathcal H$.
We show that if for some $L \geq 1$, the bounds
$\E \left[\|X_{t+1}-X_t\|_{\mathcal H}^2 \mid X_t\right]=1$ and $\|X_{t+1}-X_t\|_{\mathcal H} \leq L$
are satisfied for all times $t \geq 0$, then there is a constant $c = c(L)$ such that for $1 \leq R \leq \sqrt{t}$,
\[\pr(\|X_t-X_0\|_{\mathcal H} \leq R) \leq c \frac{R}{\sqrt{t}}\,.\]
Following [Lee-Peres, Ann.~Probab.\ 2013], this estimate has applications
to small-ball estimates for random walks on vertex-transitive graphs:
We show
that for every infinite, connected, vertex-transitive graph $G$ with bounded degree,
there is a constant $C_G > 0$ such that if $\{Z_t\}$ is the simple random walk on $G$, then
for every $\e > 0$ and $t \geq 1/\varepsilon^2$,
$$
\P\left(\vphantom{\bigoplus}\dist_G(Z_t,Z_0) \leq \e \sqrt{t}\right) \leq C_G\, \varepsilon\,,
$$
where $\dist_G$ denotes the graph distance in $G$.
\end{abstract}

\section{Introduction}

Let $\mathcal H$ be a Hilbert space with inner product $\langle \cdot,\cdot \rangle$ and norm $\|\cdot\|$, and
let $\{X_n : n \geq 0\}$ denote a discrete-time $\mathcal H$-valued martingale with respect to a filtration $\{\mathcal F_n\}$.
Suppose that for some number $L > 1$ and all $n \geq 1$, we have $\|X_{n}-X_{n-1}\| \leq L$ almost surely.
In addition,
suppose that the conditional variance $V_n = \E\left[\|X_{n}-X_{n-1}\|^2 \mid \mathcal F_{n-1}\right]$ satisfies $V_n \geq 1$ almost surely.  As discovered by A. G. \`Ershler, martingales satisfying these
conditions arise in the study of random walks on groups, as we discuss shortly.


Given these almost sure bounds on the conditional variances, one might expect some type of martingale central limit theorem to hold.  In fact, this is hopelessly false.  Such martingales can exhibit counterintuitive behavior even in the 1-dimensional case.  This phenomenon is suggested by
solutions to certain PDE arising in nonlinear filtration \cite{BS69}.

The authors of \cite{GPZ13} confirm this surprising
 behavior in the discrete setting:  For every $t \geq 1$, there is a real-valued martingale $\{X_n\}$ satisfying the above assumptions, the intial condition $X_0 = 0$, and the estimate
\[
\P(|X_t| \leq 1) \geq c t^{-\alpha},
\]
where the constants $c > 0$ and $0 < \alpha < 1/2$ are independent of $t$.  In other words, even under seemingly strong
upper and lower bounds on the increments, $X_t$ can land near the origin with
probability much greater than the order $t^{-1/2}$ achieved by simple random walk.  The moral is that, by allowing the conditional variance $V_n$ to depend on the state $(X_{n-1}, n)$, a clever controller can steer the random walk closer to small sets.

Our primary goal is to prove that, for $\mathcal H$-valued martingales,
such ``non-Gaussian'' behavior cannot happen if the sequence $\{V_n\}$ is {\em deterministic}.  We prove the small-ball estimate
\begin{equation*}
\P[\|X_n\| \leq R] \leq \frac{c_L R}{\sqrt n} \quad \mbox{for } 1 \leq R \leq \sqrt n,
\end{equation*}
when $X_0 = 0$ and the $V_n \geq 1$ are deterministic.

Note that even in this more restricted case, there is no central limit theorem.  Indeed, the increments $X_{n+1} - X_n$ can lie in different subspaces at different times.
Choosing the direction of $X_{n+1}-X_n$ allows one to
control the conditional variance of the martingale projected
onto a fixed direction.
One might suspect that this gives a controller the
ability to again substantially increase the probability
of the martingale to be near the origin at a target time $t$
as in \cite{GPZ13}.
Our main result is that this is not the case.

We use a coupling argument to reduce to the two-dimensional case, where the ratio of area to perimeter is favorable, and then argue by induction.  If one instead takes a control-theoretic approach, optimizing the increments to minimize the small-ball probability, this leads to the work of Armstrong and Zeitouni \cite{AZ14}, which we describe in more detail below.

Our main theorem includes off-diagonal estimates as well, and these are needed in the induction step.  Of course, when $\|x_0\|$ is large, this estimate is an easy consequence of Azuma's inequality.

\begin{theorem}\label{thm:martingale}
Let $\{X_n\}$ be an $\mathcal H$-valued martingale with respect to the filtration $\{\mathcal F_n\}$ and suppose
there exists a sequence of numbers $\{v_n \geq 1 : n \geq 1\}$ such that for each $n \geq 1$,
almost surely $\E \left[\|X_n - X_{n-1}\|^2 \mid \mathcal F_{n-1}\right] = v_n$ and $\|X_n-X_{n-1}\| \leq L$.
Then for every $n \geq 1$ and $1 \leq R \leq \sqrt{n}$,
we have
\[
\P(\|X_n\| \leq R \mid X_0 = x_0) \leq \frac{c_L R}{\sqrt{n}} e^{-\|x_0\|^2/(6L^2 n)}\,,
\]
where $c_L > 0$ is a constant depending only on $L$.
\end{theorem}

\medskip
\noindent
{\bf Remarks on the proof.}
The delicacy required to prove Theorem \ref{thm:martingale}
lies in the fact that one cannot uniformly dominate $\|X_n\|$
by a Gaussian in order to apply the natural induction.
Here, uniformly refers to a bound that holds simultaneously
for all martingales satisfying the Lipschitz and conditional variance conditions.
For instance, consider the two-dimensional martingale such that
$X_0=0$, and for $n \geq 0$, it holds that $\|X_{n+1}-X_n\|=1$ and $X_{n+1}-X_n$ is orthogonal to $X_n$.
In that case, $\|X_n\|^2=n$ (with probability one) for all $n \geq 0$.

Instead, our proof proceeds in two steps.  First, we prove that a martingale
cannot aim for the origin at time $n$ using a controlled trajectory, e.g. such
that $\|X_{n-t}\| \leq t^{5/8} + O(1)$ for all $1 \leq t \leq n$.
Given an uncontrolled trajectory, we break it the union of a smaller
trajectory (with fewer time steps) and an uncontrolled piece.
The smaller trajectory is bounded by induction, and the uncontrolled piece
by large deviation bounds.

The final step is
to take a union bound over a discretization of the space of all possible trajectories.
Since we have no a priori bound on the dimension of $\mathcal H$, this seems infeasible.
Here is where we employ martingale dimension reduction:  We couple our $\mathcal H$-valued
martingale to an $\mathbb R^2$-valued martingale with the same Lipschitz and conditional
variance conditions.  The union bounds thus becomes possible.

\medskip

Using the methods of \cite{LP13}, Theorem \ref{thm:martingale} can be
used to obtain a diffusive estimate for random walks on finitely-generated groups and, more generally, vertex-transitive graphs.  In particular,
the following result is proved in Section \ref{sec:groups} using Theorem \ref{thm:martingale}.

\begin{theorem}
For every infinite, locally-finite, connected, vertex-transitive graph $G$,
there is a constant $C_G > 0$ such that if $\{Z_t\}$ is the simple random walk on $G$, then
for every $\e > 0$ and every $t \geq 1/\varepsilon^2$,
$$
\P\left(\vphantom{\bigoplus}\dist_G(Z_t,Z_0) \leq \e \sqrt{t}\right) \leq C_G\, \varepsilon\,,
$$
where $\dist_G$ denotes the graph distance in $G$.
\end{theorem}

For the preceding theorem, one only requires the case $v_1=v_2=\cdots=1$ in Theorem \ref{thm:martingale}.
We also prove a related theorem for finite vertex-transitive graphs; see Theorem \ref{thm:finitegraphs} below.  In that setting,
one needs more general deterministic sequences $\{v_n\}$.

\medskip

Finally, we remark that Theorem \ref{thm:martingale} is a corollary of the main theorem of the independent and concurrent work of Armstrong and Zeitouni \cite{AZ14}.  As written, their result only applies to the case $R=1$, $v_1 = v_2 = \cdots = 1$, and $x_0 = 0$, but one expects
that they can obtain the general case by straightforward modification.

\section{Martingale small-ball probabilities}

We recall the setup of the introduction,
where $\mathcal H$ is a Hilbert space with inner product $\langle \cdot,\cdot \rangle$ and norm $\|\cdot\|$.
The process
$\{X_n : n \geq 0\}$ will denote a discrete-time $\mathcal H$-valued martingale with respect to a filtration $\{\mathcal F_n\}$.
We use the notation $\E_n\left[\cdot\right] = \E\left[ \cdot \mid \mathcal F_n\right]$ and $\P_n\left[\cdot\right] = \P\left[ \cdot \mid \mathcal F_n\right]$.
For the remainder of the section, we will assume that $\{X_n\}$ satisfies the the following two properties:
\begin{enumerate}
\item[(M1)] There is a (deterministic) sequence of numbers $\{v_n : n \geq 1\}$ such that for all $n \geq 1$, we
 have $v_n \geq 1$ and $\E_{n-1} \|X_{n}-X_{n-1}\|^2 = v_n$.
\item[(M2)] For all $n \geq 1$, $\|X_{n}-X_{n-1}\| \leq L$ almost surely.
\end{enumerate}
Note that the conjunction of (M1) and (M2) imply that $L \geq 1$.

We first prove an estimate assuming a martingale approaches the origin in a controlled manner.

\begin{lemma}\label{lem:main}
There is a universal constant $c > 0$ such that
for every $\lambda \geq 1$ and all $n \geq 1$,
the following holds:  If $\{X_n\}$ is any martingale satisfying (M1) and (M2) and $X_0=x_0$, then
\begin{equation}\label{eq:control}
\P\left[ \|X_n\| \leq 1 \textrm{ and } \|X_{n-t}\| \leq t^{5/8} + \lambda \textrm{ for all } 0 \leq t \leq n\right] \leq \frac{c L^{13} \lambda^{4/5}}{\sqrt{n}} \exp\left({\frac{-\|x_0\|^2}{2L^2 n}}\right)\,.
\end{equation}
\end{lemma}

\begin{proof}
Fix $n \geq 1$, and for $1 \leq k \leq n$, define the random variable
\[\Psi_k = \P_{n-k}\left[\|X_n\| \leq 1 \textrm{ and } \|X_{n-t}\| \leq t^{5/8} + \lambda \textrm{ for all } 0 \leq t \leq k\right]\,.\]

Let $k_0 = \lceil 30 L^{24} \lambda^{8/5} \rceil$.
Now define, for $1 \leq k \leq n$, the sequence
\begin{equation}\label{eq:seqn}
s_k = L^2 \min(k,k_0) +
\sum_{j=k_0}^{k-1} v_{n-j}\,.
\end{equation}

We will prove by induction on $k$ that for all $1 \leq k \leq n$, the
following bound holds almost surely:
\begin{equation}\label{eq:hypoth}
\Psi_k \leq e^{-\frac{\|X_{n-k}\|^2}{2 s_k}} \beta_k\,,
\end{equation}
where
\begin{equation}\label{eq:beta}
\beta_k= e^2\prod_{j=k_0+1}^k \left(1- \frac{s_j-s_{j-1}}{2 s_j} + 77 L^3 j^{-9/8}\right)\,.
\end{equation}
We take the product to be $1$ if $k \leq k_0$.

Clearly $\beta_{k} = e$ for $k \leq k_0$.  For $k > k_0$, using the fact that $\log(1+x) \leq x$ for $x > -1$, we will have
\begin{eqnarray*}
\log \beta_k &\leq& 2 - \frac12 \sum_{j=k_0+1}^{k} \frac{s_j-s_{j-1}}{s_j} +  77 L^3 \sum_{j=k_0+1}^k j^{-9/8} \\
&\leq & 2  - \frac12 \sum_{j=k_0+1}^{k} \int_{s_{j-1}}^{s_j} \frac{1}{x+L^2}\,dx  + O(1) L^3 k_0^{-1/8} \\
&\leq & O(1) - \frac12 \log \left(\frac{s_k+L^2}{s_{k_0}+L^2}\right)\,,
\end{eqnarray*}
where we have used the fact that $k_0$ is chosen large enough
so that $L^3 k_0^{-1/8}$ is bounded above by a universal constant.

From \eqref{eq:seqn}, we know that $s_{k_0} = k_0 L^2$ and from (M1), we have \[s_k \geq L^2 \min(k,k_0) + k-k_0 \geq k\,,\] hence we conclude that for all $1 \leq k \leq n$,
$$
\beta_k \leq O(1) \frac{\sqrt{k_0+1}\cdot L}{\sqrt{k}} \leq O(1) \frac{L^{13} \lambda^{4/5}}{\sqrt{k}}\,.
$$
Combining this with \eqref{eq:hypoth} and using the fact that $s_n \leq L^2 n$ yields \eqref{eq:control}.  (Observe
that $\Psi_n$ is precisely what we are trying to bound in \eqref{eq:control}.)

\medskip

Thus we are now left to prove \eqref{eq:hypoth} by induction on $k$.
First consider the case $1 \leq k \leq k_0$.  If $\|X_{n-k}\| > k L + 1$, then $\Psi_{k}=0$
almost surely,
because the Lipschitz bound (M2) implies that the chain can move at most $kL$ distance
in $k$ steps.
We may also assume that $\|X_{n-k}\| > 1$, else the bound is trivially true.  In particular, if we define the event
\begin{equation*}
\mathcal E = \{ 1 < \| X_{n-k} \| \leq k L + 1 \},
\end{equation*}
then
\begin{equation*}
\P\left(\Psi_k \mid \mathcal E^c \right) \leq e^2 \cdot e^{\frac{-\|X_{n-k}\|^2}{2 k L^2}}.
\end{equation*}

On the event $\mathcal E$,  we can apply
Azuma's inequality
to the $1$-dimensional martingale
\[\left\{\left\langle X_{n-k}-X_t,
\tfrac{X_{n-k}}{\|X_{n-k}\|}
\right\rangle : t=n-k, n-k+1, \ldots, n\right\}\] to conclude
that almost surely,
\begin{eqnarray*}
\P\left(\Psi_{k} \mid \mathcal E\right) &\leq&
\P\left(\vphantom{\bigoplus} \|X_n\| \leq 1 \mid \mathcal E\right) \\
&\leq &
\P\left(\left\langle \frac{X_{n-k}}{\|X_{n-k}\|}, X_{n-k} - X_n\right\rangle > \|X_{n-k}\|-1 \mid \mathcal E\right)
\\
&\leq &
e^{-\frac{(\|X_{n-k}\|-1)^2}{2 k L^2}} = e^{\frac{-\|X_{n-k}\|^2}{2 k L^2}} e^{\frac{\|X_{n-k}\|-1/2}{n L^2}} \leq e^2 \cdot e^{\frac{-\|X_{n-k}\|^2}{2 k L^2}}\,,
\end{eqnarray*}
where the last inequality uses our assumption that $\|X_{n-k}\| \leq k L + 1$ and the fact that $L \geq 1$.
Thus the bound \eqref{eq:hypoth} is satisfied since for $k \leq k_0$, we have $\beta_{k} = e^2$ and $s_k = k L^2$ (recalling \eqref{eq:seqn}).

\medskip

We are thus left to prove \eqref{eq:hypoth} by induction for $k > k_0$.
We may assume that
\begin{equation}\label{eq:assume}
\|X_{n-k}\| \leq k^{5/8} + \lambda\,,
\end{equation}
since otherwise $\Psi_k = 0$.

\medskip

Now we will use the inductive hypothesis to calculate:
\begin{align}
\Psi_{k} &= \E_{n-k} \pr_{n-k+1}\left[ \|X_n\| \leq 1 \textrm{ and } \|X_{n-t}\| \leq t^{5/8} + \lambda \textrm{ for all } 0 \leq t \leq k\right] \nonumber\\
&\leq
\E_{n-k} \pr_{n-k+1}\left[ \|X_n\| \leq 1 \textrm{ and } \|X_{n-t}\| \leq t^{5/8} + \lambda\textrm{ for all } 0 \leq t \leq k-1\right]\nonumber
\\
&= \E_{n-k} \left[\Psi_{k-1}\right] \nonumber\\
& \leq \beta_{k-1} \E_{n-k} \left[e^{-\frac{\|X_{n-k+1}\|^2}{2 s_{k-1}}}\right], \nonumber
\end{align}
where in the final line we have employed the inductive hypothesis.
Observe that here we have used the fact that $\beta_{k-1}$ is a constant; indeed,
this is where we employ our assumption that the sequence $\{v_n\}$ is deterministic.

Letting $D = X_{n-k+1}-X_{n-k}$ and using the preceding inequality, we have
\begin{align}
\Psi_k
&\leq \beta_{k-1} \E_{n-k} \left[e^{-\frac{\|X_{n-k}+D\|^2}{2 s_{k-1}}}\right] \nonumber\\
&= \beta_{k-1}  e^{-\frac{\|X_{n-k}\|^2}{2 s_k}} \cdot \E_{n-k}\left[\exp\left(\frac{-\|D\|^2}{2 s_{k-1}} - \frac{\langle X_{n-k},D\rangle}{s_{k-1}} +
\frac{\|X_{n-k}\|^2}{2} \frac{s_{k-1}-s_k}{s_{k-1} s_k}\right)\right]\,.\label{eq:final}
\end{align}
Observe that, by \eqref{eq:assume}, and assumptions (M1) and (M2),
the three terms inside the exponential (almost surely) have
their respective
magnitudes bounded by
\begin{equation}\label{eq:mags}
\frac{L^2}{2(k-1)},\quad \frac{(k^{5/8}+\lambda)L}{k-1},\quad \frac{(k^{5/8}+\lambda)^2 L^2}{2k(k-1)}\,.
\end{equation}
For $k \geq k_0 \geq 30 L^{24} \lambda^{8/5}$,
each of these terms is bounded by $\min(1/3,2L k^{-3/8})$.

We require the following basic approximation.
\begin{lemma}\label{lem:approx}
If $y \leq 1$, then
$$
e^y - (1+y+y^2/2) \leq |y|^3.
$$
\end{lemma}

Using (M2), we have $\|D\| \leq L$ almost surely.  In conjunction with \eqref{eq:assume} and the bounds \eqref{eq:mags}, we may apply Lemma \ref{lem:approx}
to write
\begin{align}
\E_{n-k}&\left[\exp\left(\frac{-\|D\|^2}{2 s_{k-1}} - \frac{\langle X_{n-k},D\rangle}{s_{k-1}} +
\frac{\|X_{n-k}\|^2}{2} \frac{s_{k-1}-s_k}{s_{k-1} s_k}\right)\right] \nonumber \\
&\leq
1 - \E_{n-k} \frac{\|D\|^2}{2 s_{k-1}} + \frac{\|X_{n-k}\|^2}{2} \frac{s_{k-1}-s_k}{s_{k-1} s_k} + \E_{n-k} \frac{\langle X_{n-k}, D\rangle^2}{2 s_{k-1}^2} + (4+72) L^3 k^{-9/8}\,,\label{eq:latter}
\end{align}
where we have also used the martingale property $\E_{n-k} D=0$.
The error term multiplied by $4$ comes from bounding the remaining quadratic terms using \eqref{eq:mags},
and the term multiplied by $72$ arises from the cubic error in Lemma \ref{lem:approx}.

Now using the fact that $\E_{n-k} \|D\|^2 = s_k-s_{k-1}$ (which follows from (M1) and the definition \eqref{eq:seqn}), along
with Cauchy-Schwarz, we can bound \eqref{eq:latter} by
\begin{align}
1 - & \frac{s_k-s_{k-1}}{2s_{k-1}} + \frac{\|X_{n-k}\|^2}{2} \frac{s_{k-1}-s_k}{s_{k-1} s_k} + \frac{\|X_{n-k}\|^2}{2} \frac{s_k-s_{k-1}}{s_{k-1}^2} + 76 L^3 k^{-9/8} \nonumber\\
&=
1 - \frac{s_k-s_{k-1}}{2s_{k-1}} + \frac{\|X_{n-k}\|^2}{2} \frac{(s_{k}-s_{k-1})^2}{s^2_{k-1} s_k} + 76 L^3 k^{-9/8} \nonumber\\
&\leq
1 - \frac{s_k-s_{k-1}}{2s_{k-1}} +  77 L^3 k^{-9/8}\,,\label{eq:explicit}
\end{align}
where in the final line we have used \eqref{eq:assume}, the fact that $s_k-s_{k-1} \leq L$ by (M2),
the fact that $s_k s_{k-1}^2 \geq k (k-1)^2$ by (M1),
and our
assumption that $k \geq k_0 \geq 30 L^{24} \lambda^{8/5}$.

Recalling \eqref{eq:final}, we have verified that almost surely
\[
\Psi_k \leq \beta_{k-1} e^{-\frac{\|X_{n-k}\|^2}{2s_k}} \left(1 - \frac{s_k-s_{k-1}}{2s_{k-1}} + 77 L^3 k^{-9/8}\right) = \beta_k e^{-\frac{\|X_{n-k}\|^2}{2s_k}}\,.
\]
This completes the proof of \eqref{eq:hypoth} by induction.
\end{proof}

We will use the preceding estimate to control the small-ball probability.
Before that, we observe that it suffices to prove a bound for $\mathbb R^2$-valued martingales.
The following dimension reduction lemma is a special case
of the continuous-time version proved in \cite{KS91}.
We include a proof here for the convenience of the reader.
A similar exposition of the discrete case appears in
\cite[Prop. 5.8.3]{KW92}.

\begin{lemma}\label{lem:KS}
Let $\{N_t\}$ be an $\mathcal H$-valued martingale.
Then there exists an $\mathbb R^2$-valued martingale $\{M_t\}$ such that
for any time $t \geq 0$,
$\|M_t\|_2= \|N_t\|_2$ and $\|M_{t+1}-M_t\|_2 = \|N_{t+1}-N_t\|_2$.
\end{lemma}

\begin{proof}
We prove the claim by induction on $n$. The case $n=0$ is trivial.
Suppose now we can construct $\{M_t\}_{t \leq n}$ successfully based on $\{N_t\}_{t \leq n}$. We wish to specify the value of $M_{n+1}$ given $\{N_t\}_{t\leq n+1}$ and $\{M_t\}_{t\leq n}$ such that the required conditions hold.

In the generic case, there exist two distinct points in $x_1, x_2 \in \mathbb R^2$ satisfying
\begin{equation}\label{eq-norm-product}\|N_{n+1}\| = \|x_i\|, \mbox{ and } \langle N_n, N_{n+1}\rangle  = \langle M_n, x_i\rangle,\end{equation}
for $i=1,2$.
(One can see them as the intersections of a circle and a line.) Denoting these two points by $M_{n+1}^{(1)}$ and $M_{n+1}^{(2)}$,
we now let $M_{n+1}$ be $M_{n+1}^{(1)}$ (resp., $M_{n+1}^{(2)}$) with probability $\frac{1}{2}$. It is clear that $\|N_{n+1}\| = \|M_{n+1}\|$. Recalling \eqref{eq-norm-product} and using the induction hypothesis $\|N_n\| = \|M_n\|$, we also infer that
\begin{align*}\|M_{n+1} - M_n\|^2 &= \|M_{n+1}\|^2 - 2\langle M_{n+1}, M_n\rangle  + \|M_n\|^2 \\
&= \|N_{n+1}\|^2 - 2\langle N_{n+1}, N_n\rangle  + \|N_n\|^2 \\
&= \|N_{n+1} - N_n\|^2~.\end{align*}

It remains to prove that $\E[M_{n+1} \mid M_1, \ldots, M_n] = M_n$. To this end, it suffices to show
\begin{align}
\E[\langle M_{n+1}, M_n^{\bot} \rangle \mid M_1, \ldots, M_n] &= 0, \label{eq-condition-1}\\
\E[\langle M_{n+1}, M_n\rangle  \mid M_1, \ldots, M_n] &= \|M_n\|^2, \label{eq-condition-2}
\end{align}
where $M_n^{\bot}$ is a unit vector with $\langle M_n^{\bot}, M_n\rangle =0$. Equality \eqref{eq-condition-1} follows by our uniform random choice of $M_n$ over $\{M_{n+1}^{(1)}, M_{n+1}^{(2)}\}$. Since $\{N_t\}$ is a martingale, we have that $\E\{\langle N_{n+1}, N_n \rangle  \mid N_1,\ldots, N_n\} = \|N_n\|^2$. Combined with \eqref{eq-norm-product} and our choice of $M_{n+1}$, we obtain \eqref{eq-condition-2} as required.

In the degenerate case when $N_n$ and $N_{n+1}$ are proportional, there is a unique solution to (\ref{eq-norm-product}), and we just let $M_{n+1}$ be that unique point.
In the case when $N_n=0$ but $N_{n+1} \neq 0$, there are infinitely many solutions,
one can pick out two symmetric ones and let $M_{n+1}$ be uniformly random over those two points.
\end{proof}

We now proceed to our first small-ball estimate.

\begin{theorem}\label{thm:smallball}
Assume that $\{X_n\}$ is an $\mathcal H$-valued martingale satisfying (M1) and (M2).  If $X_0=x_0$, then for any $n \geq 1$,
\[
\P(\|X_n\| \leq 1) \leq \frac{O(L^{20})}{\sqrt{n}} e^{-\frac{\|x_0\|^2}{3L^2 n}}\,.
\]
\end{theorem}

\begin{proof}
By Lemma \ref{lem:KS}, we may assume that $\{X_n\}$ takes values in $\mathbb R^2$.
By induction on $n$, we will prove that
\begin{equation}\label{eq:toprove}
\P(\|X_n\| \leq 1) \leq \frac{B}{\sqrt{n}} e^{-\frac{\|x_0\|^2}{3 L^2 n}}\,\,
\end{equation}
for some number $B \leq O(L^{20})$ to be chosen later.
The case $n=1$ is trivial as long as $B \geq e$, since the left-hand side is 0 for $\|x_0\| > L+1$ (by (M2)).

Also observe that Azuma's inequality applied
to the $1$-dimensional martingale $\{\langle x_0, x_0 - X_t\rangle\}$
 implies that
\[
\P(\|X_n\| \leq 1) \leq e^{-\frac{\max(0,\|x_0\|-1)^2}{2 L^2 n}}\,.
\]
If $\|x_0\| > 3L\sqrt{n \log n}$, then
\[
e^{-\frac{\max(0,\|x_0\|-1)^2}{2 L^2 n}} \leq \frac{B}{\sqrt{n}} e^{\frac{-\|x_0\|^2}{3 L^2 n}}\,,
\]
as long as $B > 0$ is a sufficiently large constant.
Thus we may assume that
\begin{equation}\label{eq:Xnorm}
\|x_0\| \leq 3L\sqrt{n \log n}\,.
\end{equation}

Let $k_0 \geq 1$ be a number to be chosen later and put $\lambda = k_0 L$.
We may decompose
\begin{align}
\P\left(\|X_n\| \leq 1\right) & \leq \P\left(\|X_n\| \leq 1 \textrm{ and } \|X_{n-t}\| \leq t^{5/8} + \lambda \textrm{ for } 0 \leq t \leq n\right) \nonumber\\
&\ \ \ + \sum_{k=k_0}^n \P\left(\|X_n\| \leq 1 \textrm{ and } \|X_{n-k}\| > k^{5/8}+1\right) \nonumber\\
& \leq \frac{c L^{64/5} k_0^{4/5}}{\sqrt{n}} e^{-\frac{\|x_0\|^2}{2L^2 n}} + \sum_{k=k_0}^n \P\left(\|X_n\| \leq 1 \textrm{ and } \|X_{n-k}\| > k^{5/8}+1\right) \,, \label{eq:part1}
\end{align}
where we have bounded the
 first term using Lemma \ref{lem:main}.

Note that if $\|X_n\| \leq 1$ then by (M2), we must have $\|X_{n-k}\| \leq kL+1$.
Let $N_k$ denote a $1$-net in the Euclidean disk of radius $kL+1$ about $0$, and observe that
$|N_k| \leq 4 (kL+1)^2$.  Thus we have
\begin{align}
\P&\left(\|X_n\| \leq 1 \textrm{ and } \|X_{n-k}\| > k^{5/8}+1\right)\nonumber\\
&=
\P\left(\|X_n\| \leq 1 \textrm{ and } kL+1 \geq \|X_{n-k}\| > k^{5/8}+1\right) \nonumber \\
&\leq
\P\left(\|X_n\| \leq 1 \mid kL+1 \geq \|X_{n-k}\| > k^{5/8}+1\right) \P\left(\|X_{n-k}\| \leq kL+1\right) \nonumber \\
& \leq \P\left(\|X_n\| \leq 1 \mid kL+1 \geq \|X_{n-k}\| > k^{5/8}+1\right) \sum_{y \in N_k} \P(\|X_{n-k}-y\| \leq 1) \nonumber \\
&\leq   \P\left(\|X_n\| \leq 1 \mid kL+1 \geq \|X_{n-k}\| > k^{5/8}+1\right) \sum_{y \in N_k} \frac{B}{\sqrt{n-k}} e^{-\frac{\|x_0-y\|^2}{3L^2(n-k)}} \nonumber \\
&\leq \P\left(\|X_n\| \leq 1 \mid kL+1 \geq \|X_{n-k}\| > k^{5/8}+1\right) \frac{B |N_k|}{\sqrt{n-k}} e^{-\frac{\max(0,\|x_0\|-(kL+1))^2}{3L^2(n-k)}},\label{eq:part2}
\end{align}
where in the third line we have used the inductive hypothesis,
and in the final line a union bound.

We may then apply Azuma's inequality to the $1$-dimensional martingale
\[\left\{\left\langle  X_{t}-X_{n-k},\tfrac{X_{n-k}}{\|X_{n-k}\|}\right\rangle : t =n-k, n-k+1, \ldots, n\right\}\]
to conclude that
\begin{align}
\P&\left(\|X_n\| \leq 1 \,\big|\, kL+1 \geq \|X_{n-k}\| > k^{5/8}+1\right) \nonumber \\
& \leq \P\left(\vphantom{\bigoplus}\langle X_n-X_{n-k}, \tfrac{X_{n-k}}{\|X_{n-k}\|} \rangle \geq k^{5/8} \,\Big|\, kL+1 \geq \|X_{n-k}\| > k^{5/8}+1\right) \nonumber \\
& \leq 
\P\left(\vphantom{\bigoplus}\langle X_{n}-X_{n-k}, 
 \tfrac{X_{n-k}}{\|X_{n-k}\| }
 \rangle \geq k^{5/8} \mid X_{n-k} \right)
\nonumber \\
 & \leq e^{-k^{1/4}/(2L^2)}\,.\label{eq:part3}
\end{align}

Combining \eqref{eq:part1}, \eqref{eq:part2}, \eqref{eq:part3}, and using $|N_k| \leq 4(kL+1)^2$ yields
\begin{align*}
\P(\|X_n\| \leq 1) & \leq  \frac{cL^{64/5} k_0^{4/5}}{\sqrt{n}} e^{-\frac{\|x_0\|^2}{2L^2 n}} \\
&\,\,\,\,\,\,+ B \sum_{k=k_0}^n \frac{4(kL+1)^2}{\sqrt{n-k}}
\exp\left(\frac{-k^{1/4}}{2L^2} - \frac{\max(0,\|x_0\|-(kL+1))^2}{3L^2(n-k)} \right)\,.
\end{align*}

Our goal is now to prove that there is a universal constant $\alpha > 0$ (in particular,
$\alpha$ will not depend on $B$) such that
\begin{equation}\label{eq:left-to-prove}
\exp\left(\frac{-k^{1/4}}{4L^2} - \frac{\max(0,\|x_0\|-(kL+1))^2}{3L^2(n-k)} \right) \leq \alpha \exp\left(\frac{-\|x_0\|^2}{3L^2 n}\right)\,.
\end{equation}
Plugging this estimate into the preceding inequality yields
\begin{align*}
\P(\|X_n\| \leq 1) & \leq  \frac{cL^{64/5} k_0^{4/5}}{\sqrt{n}} e^{-\frac{\|x_0\|^2}{2L^2 n}}
+ B \exp\left(\frac{-\|x_0\|^2}{3L^2 n}\right) \left[\alpha \sum_{k=k_0}^n \frac{4(kL+1)^2}{\sqrt{n-k}}
\exp\left(\frac{-k^{1/4}}{4L^2}\right)\right].
\end{align*}

By choosing $k_0 \asymp L^{9}$ large enough (depending on $\alpha$), the sum in brackets is at most $\frac{1}{2 \sqrt{n}}$.
Indeed, one can choose $k_0$ such that the value is at most
\[
\alpha \sum_{k=k_0}^n \frac{e^{-k^{1/36}/8}}{\sqrt{n-k}}\,.
\]
This sum is dominated by its first term which can be made
arbitrarily small by an appropriate choice of $k_0$.

Fixing this value of $k_0$ and setting $B = 2 c L^{64/5} k_0^{4/5} \leq O(L^{20})$
shows that \[\P(\|X_n\| \leq 1) \leq \frac{B}{\sqrt{n}} e^{-\frac{\|x_0\|^2}{3L^2 n}}\,,\]
completing the proof of \eqref{eq:toprove} by induction.  Thus we are left to prove \eqref{eq:left-to-prove} for $k \geq k_0$.

\medskip

\noindent
{\bf Case I:} $\|x_0\| \leq kL+1$.

\medskip

In this case, we need to show that $\exp(\frac{-k^{1/4}}{4L^2}) \leq \alpha \exp(\frac{-\|x_0\|^2}{3L^2 n})$.
We may assume that $\|x_0\| > L\sqrt{n}+1$, else the inequality holds trivially for some $\alpha = O(1)$.
In particular, we may assume that $k > \sqrt{n}$.  But then our assumption \eqref{eq:Xnorm} that $\|x_0\| \leq 3 L \sqrt{n \log n}$
shows that the inequality holds for some $\alpha = O(1)$ (with room to spare).

\medskip
\noindent
{\bf Case II:} $\|x_0\| > kL+1$.

\medskip

In this case, it suffices to argue that
\[
\exp\left(\frac{-k^{1/4}}{4L^2} - \frac{(\|x_0\|-(kL+1))^2}{3L^2(n-k)} + \frac{\|x_0\|^2}{3L^2 n} \right) \leq O(1)\,.
\]
Expanding the square, we see that it is enough to show
\begin{equation}\label{eq:finally}
\exp\left(\frac{-k^{1/4}}{4 L^2} + \frac{2 (kL+1)\|x_0\|}{3L^2(n-k)}\right) \leq O(1)\,.
\end{equation}
Recalling \eqref{eq:Xnorm} that $\|x_0\| \leq 3L\sqrt{n \log n}$,
we have $k \leq 3\sqrt{n \log n}$.  Thus the positive term in \eqref{eq:finally} is $O(1)$
unless $k \geq \sqrt{n/\log n}$.  But if $\sqrt{n/\log n} \leq k \leq 3\sqrt{n \log n}$, then we have
\[
\frac{k^{1/4}}{4L^2} \geq \frac{2(kL+1) 3L\sqrt{n \log n}}{3L^2 n} \geq \frac{2(kL+1) \|x_0\|}{3L^2 (n-k)}\,
\]
where we have additionally used the fact that $k \geq k_0$ and $k_0 \asymp L^9$ is chosen large enough.
We have thus verified \eqref{eq:finally}, completing the proof.
\end{proof}

Finally, we extend Theorem \ref{thm:smallball} to larger radii.

\begin{theorem}\label{thm:largeballs}
Assume that $\{X_n\}$ is an $\mathcal H$-valued martingale satisfying (M1) and (M2),
with $X_0=x_0$.
Then for any $n \geq 1$ and  $1 \leq R \leq \sqrt{n}$, we have
\[
\P(\|X_n\| \leq R) \leq O(L^{20}) \frac{R}{\sqrt{n}} e^{-\frac{\|x_0\|^2}{6 L^2 n}}\,.
\]
\end{theorem}

\begin{proof}
Consider a martingale $\{Y_t\}$ defined as follows.  For $1 \leq t \leq n$, we set $Y_t=X_t$.
Let $m = \lfloor R^2\rfloor$.
For $n < t \leq n+m$, put
$$
Y_t = X_n + \frac{X_n}{\|X_n\|} \sum_{j=1}^{t-n} \varepsilon_j\,,
$$
were $\{\varepsilon_j\}$ are i.i.d. signs independent of $\{X_n\}$.  Then the martingale $\{Y_t\}_{t=0}^{n+m}$ satisfies
assumptions (M1) and (M2)
hence by Theorem \ref{thm:smallball},
\begin{equation}\label{eq:radone}
\P(\|Y_{n+m}\| \leq 1) \leq \frac{O(L^{20})}{\sqrt{n+m}} e^{-\frac{\|x_0\|^2}{3L^2(n+m)}}\,.
\end{equation}
On the other hand, since simple random walk satisfies a local CLT, there is a constant $c > 0$ such that
\[
\P(\|Y_{n+m}\| \leq 1) \geq \frac{c}{R} \P(\|X_n\| \leq R)\,.
\]
Combining this with \eqref{eq:radone} yields the desired result.
\end{proof}

\section{Random walks on vertex-transitive graphs}
\label{sec:groups}

A primary application of our small-ball estimate is to random walks on vertex-transitive graphs.
We will use $\dist_G$ to denote the shortest-path metric on a graph $G$.

\begin{theorem}[Diffusive random walks]
\label{thm:diffusive}
For every infinite, locally-finite, connected, vertex-transitive graph $G$,
there is a constant $C_G > 0$ such that if $\{Z_t\}$ is the random walk on $G$, then
for every $\e > 0$ and every $t \geq 1/\varepsilon^2$,
$$
\P\left(\vphantom{\bigoplus}\dist_G(Z_t,Z_0) \leq \e \sqrt{t}\right) \leq C_G\, \varepsilon.
$$
\end{theorem}

This should be compared to the result of the first two authors \cite{LP13}
which shows that this property holds for an average $t$, i.e.
\[
\frac{1}{t} \sum_{s=0}^t \P\left(\dist_G(Z_0,Z_s) \leq \varepsilon \sqrt{t}\right) \leq C_G\,\varepsilon\,.
\]

If $G$ is non-amenable, then the random walk has spectral radius $\rho < 1$,
so \[\P(\dist_G(Z_0,Z_t) \leq \e \sqrt{t}) \leq d^{\e \sqrt{t}} \rho^t\,.\]
(See, for example, \cite{Woess00}.)  The latter quantity is at most $C_{\rho,d}\, \e$ for $t \geq 1/\e^2$.
Thus Theorem \ref{thm:diffusive} follows from an analysis of the amenable case.

\begin{theorem}
\label{thm:nonamenable}
If $G$ is a $d$-regular, infinite, connected, vertex-transitive graph that is also amenable and $\{Z_t\}$ is
the random walk on $G$, then the following holds.
For any $\e > 0$ and $t \geq 1/\e^2$,
\[
\P\left(\dist_G(Z_0,Z_t) \leq \e \sqrt{t/d}\right) \leq K d^{10} \varepsilon\,,
\]
where the constant $K > 0$ is universal.
\end{theorem}

\begin{proof}
Suppose that $G$ has vertex set $V$.  Let $\Aut(G)$ denote the automorphism group of $G$.
By \cite[Thm. 3.1]{LP13}, there is a
Hilbert space $\mathcal H$ on which $\Aut(G)$ acts by isometries, and a non-constant equivariant harmonic mapping
$\Psi : V \to \mathcal H$.  In other words, one has $\sigma \Psi(x) = \Psi(\sigma x)$ for all $\sigma \in \Aut(G)$ and $x \in V$.
(In fact, one can take $\mathcal H = \ell^2(V)$ and then $\Aut(G)$ acts on $\ell^2(V)$ by permutation of the coordinates.)

In particular, for any pair of vertices $x,y \in V$, we have
\[
\E \left[\|\Psi(Z_0)-\Psi(Z_1)\|^2 \mid Z_0 = x\right] = \E \left[\|\Psi(Z_0)-\Psi(Z_1)\|^2 \mid Z_0 = y\right].
\]
Since $\Psi$ is non-constant, we may
normalize $\Psi$ so that $\E \|\Psi(Z_0)-\Psi(Z_1)\|^2=1$.
Writing \[\E\, \|\Psi(Z_0)-\Psi(Z_1)\|^2 = \frac{1}{d} \sum_{y : \{y,Z_0\} \in E} \|\Psi(Z_0)-\Psi(y)\|^2=1\,,\]
one concludes that $\Psi$ is a $\sqrt{d}$-Lipschitz mapping from $(V,\dist_G)$ into $\mathcal H$.
Additionally, since $\Psi$ is harmonic, the process $\{X_t=\Psi(Z_t)\}$ is a martingale
to which Theorem \ref{thm:largeballs} applies, with $L=\sqrt{d}$.
Thus we have
\[
\P\left(\dist(Z_0,Z_t) \leq \e \sqrt{t/d}\right) \leq \P\left(\|X_0-X_t\| \leq \e \sqrt{t}\right) \leq O(1) d^{10} \e\,.\qedhere
\]
\end{proof}

One can make a similar statement about random walks on finite vertex-transitive graphs, up to the relaxation time.
(The method of proof is also from \cite{LP13}.)

\begin{theorem}
\label{thm:finitegraphs}
Suppose $G=(V,E)$ is a finite, connected, vertex-transitive $d$-regular graph, and
$\lambda$ denotes the second-largest eigenvalue of the transition matrix of
the random walk on $G$.  Then for every $t \leq (1-\lambda)^{-1}$ and
every $\e \geq 1/\sqrt{t}$,
\[
\P\left(\dist_G(Z_0,Z_t) \leq \e\sqrt{t/d}\right) \leq O(d^{10}) \e\,.
\]
\end{theorem}

\begin{proof}
We may assume that $\lambda \geq \frac12$, else the statement is vacuously true.
Let $P$ be the transition matrix of the random walk on $G$, and
let $\psi : V \to \mathbb R$ be an eigenfunction of $P$ with eigenvalue $\frac12 \leq \lambda < 1$ and norm-squared $\sum_{u \in V} \psi(u)^2 = 1$.
First, observe that
\begin{eqnarray}
\sum_{u \in V} \frac{1}{d} \sum_{v : \{u,v\} \in E} |\lambda \psi(u)-\psi(v)|^2 &=& \sum_{u \in V} (1+\lambda^2) \psi(u)^2 - 2 \lambda \sum_{u \in V} \psi(u) \frac{1}{d} \sum_{v:\{u,v\} \in E} \psi(v) \nonumber \\
&=& \langle \psi, ((1+\lambda^2)I - 2\lambda P)\psi \rangle \nonumber \\
&=& 1+\lambda^2 - 2\lambda^2 \nonumber \\
&=& 1 - \lambda^2\,.\label{eq:eigexpand}
\end{eqnarray}

Consider the automorphism group $\Aut(G)$ of $G$ and
define the map $\Psi : V \to \mathbb R^{|\Aut(G)|}$ by
\[
\Psi(v) = \frac{n}{|\Aut(G)|} \cdot \frac{\left(\psi(\sigma v)\right)_{\sigma \in \Aut(G)}}{1-\lambda^2}\,.
\]
%
We claim that the process $\{\lambda^{-t} \Psi(Z_t)\}$ is a martingale.
This follows from the fact that $\{\lambda^{-t} \psi(Z_t)\}$ is a martingale, which
can easily be checked:
\[
\E [\lambda^{-t-1} \psi(Z_{t+1}) \mid Z_t] = \lambda^{-t-1} (P\psi) (Z_t) = \lambda^{-t} \psi(Z_t)\,.
\]
Next, observe that
\begin{eqnarray}
\E \left[\|\lambda^{-t-1} \Psi(Z_{t+1}) - \lambda^{-t} \Psi(Z_t)\|^2 \mid Z_t\right] \nonumber &=&
\lambda^{-2(t-1)} \E\left[\|\lambda \Psi(Z_t) - \Psi(Z_{t+1})\|^2 \mid Z_t\right] \nonumber \\
&=&
\lambda^{-2(t-1)} (1-\lambda^2)^{-1} \sum_{u \in V} \frac{1}{d} \sum_{v : \{u,v\} \in E} |\lambda \psi(u)-\psi(v)|^2 \nonumber \\
&=&
\lambda^{-2(t-1)}\,, \label{eq:onestep}
\end{eqnarray}
where the final line uses \eqref{eq:eigexpand}.

From this, we learn two things.
First, for $t \geq 1$,
\begin{eqnarray*}
\E \left[\|\lambda^{-t-1} \Psi(Z_{t+1}) - \lambda^{-t} \Psi(Z_t)\|^2 \mid Z_t\right] &=& \lambda^{-2(t-1)} \geq 1\,.
\end{eqnarray*}
Secondly, we have a Lipschitz condition for small times:
Consider $t \leq (1-\lambda)^{-1}$ and $\{u,v\} \in E$.  Then using \eqref{eq:onestep}, we have
\begin{eqnarray*}
\|\lambda^{-t-1} \Psi(u) - \lambda^{-t} \Psi(v)\| &\leq& \sqrt{d} \cdot \left(\E \left[\|\lambda^{-t-1} \Psi(Z_{t+1}) - \lambda^{-t} \Psi(Z_t)\|^2 \mid Z_t=u\right]\vphantom{\bigoplus}\right)^{1/2}
\\
&\leq& \lambda^{-(t-1)} \leq \sqrt{d} \lambda^{\frac{-1}{1-\lambda}}
\leq 4\sqrt{d}\,,
\end{eqnarray*}
where we have used the fact that $1 \geq \lambda \geq \frac12$.

Now applying Theorem \ref{thm:largeballs} to the martingale $\{X_t = \lambda^{-t} \Psi(Z_t)\}$ for times $t \leq (1-\lambda)^{-1}$,
we see that
\[
\P\left(\dist_G(Z_0,Z_t) \leq \e\sqrt{t/d}\right) \leq \P(\|X_t-X_0\| \leq 4\e\sqrt{t}) \leq O(d^{10}) \e\,.\qedhere
\]
\end{proof}

\bibliographystyle{alpha}
\bibliography{smallball}

\begin{thebibliography}{GPZ13}

\bibitem[AZ14]{AZ14}
Scott~N. Armstrong and Ofer Zeitouni.
\newblock Local asymptotics for controlled martingales, 2014.
\newblock Preprint: arXiv:1402.2402.

\bibitem[BS69]{BS69}
G.~I. Barenblatt and G.~I. Sivashinskii.
\newblock Self-similar solutions of the second kind in nonlinear filtration.
\newblock {\em J. Appl. Math. Mech.}, 33:836--845 (1970), 1969.

\bibitem[GPZ13]{GPZ13}
Ori {Gurel-Gurevich}, Yuval Peres, and Ofer Zeitouni.
\newblock Localization for controlled random walks and martingales, 2013.
\newblock Preprint: arXiv:1309.4512.

\bibitem[KS91]{KS91}
Olav Kallenberg and Rafa{\l} Sztencel.
\newblock Some dimension-free features of vector-valued martingales.
\newblock {\em Probab. Theory Related Fields}, 88(2):215--247, 1991.

\bibitem[KW92]{KW92}
Stanis{\l}aw Kwapie{\'n} and Wojbor~A. Woyczy{\'n}ski.
\newblock {\em Random series and stochastic integrals: single and multiple}.
\newblock Probability and its Applications. Birkh\"auser Boston Inc., Boston,
  MA, 1992.

\bibitem[LP13]{LP13}
James~R. Lee and Yuval Peres.
\newblock Harmonic maps on amenable groups and a diffusive lower bound for
  random walks.
\newblock {\em Ann. Probab.}, 41(5):3392--3419, 2013.

\bibitem[Woe00]{Woess00}
Wolfgang Woess.
\newblock {\em Random walks on infinite graphs and groups}, volume 138 of {\em
  Cambridge Tracts in Mathematics}.
\newblock Cambridge University Press, Cambridge, 2000.

\end{thebibliography}

\end{document}